\documentclass[10pt,a4paper,english]{amsart}
\usepackage{babel}
\usepackage[numbers]{natbib}
\usepackage{amsfonts, amsmath, wasysym}
\usepackage{amssymb, amsthm}
\usepackage{mathrsfs}
\usepackage{verbatim}
\newfont{\ffont}{cmr10}

\allowdisplaybreaks

\newcommand{\cH}{\mathcal{H}}
\newcommand{\cA}{\mathcal{A}}

\newcommand{\cE}{\mathcal{E}}

\newcommand{\cO}{\mathcal{O}}

\newcommand{\cM}{\mathcal{M}}
\newcommand{\cN}{\mathcal{N}}

\newcommand{\id}{\mathbf{1}}

\newcommand{\tr}{\tau}
\newcommand{\al}{\alpha}

\newcommand{\si}{\sigma}
\newcommand{\Si}{\Sigma}
\newcommand{\R}{\mathbb{R}}
\newcommand{\C}{\mathbb{C}}

\newcommand{\Z}{\mathbb{Z}}

\newcommand{\ti}{\times}

\newcommand{\xspace}{\hbox{\kern-2.5pt}}

\begin{document}

\newtheorem{definition}{Definition}[section]
\newtheorem{theorem}[definition]{Theorem}
\newtheorem{conjecture}[definition]{Conjecture}
\newtheorem{proposition}[definition]{Proposition}
\newtheorem{corollary}[definition]{Corollary}
\newtheorem{remark}[definition]{Remark}
\newtheorem{lemma}[definition]{Lemma}
\newtheorem{example}{Example}[section]
\newtheorem{exercise}{Exercise}[section]

\title[Weak-type interpolation for noncommutative maximal operators]{Weak-type interpolation for noncommutative maximal operators}
\author{Sjoerd Dirksen}

\address{Universit\"{a}t Bonn\\
Hausdorff Center for Mathematics\\
Endenicher Allee 60\\
53115 Bonn\\
Germany} \email{sjoerd.dirksen@hcm.uni-bonn.de}

\thanks{This research was supported by the Hausdorff Center for Mathematics}
\maketitle

\begin{abstract}
We prove a Boyd-type interpolation result for noncommutative maximal operators of restricted weak type. Our result positively answers an open question in \cite{BCO12}. As a special case, we find a restricted weak type version of the noncommutative Marcinkiewicz interpolation theorem, due to Junge and Xu, with interpolation constant of optimal order.
\end{abstract}

\section{Introduction}

To any sequence $(T_n)_{n\geq 1}$ of sublinear operators on a space of measurable functions one can associate the operator
$$Tf(t)=\Big(\sup_{n\geq 1}T_n\Big)(f)(t):=\sup_{n\geq 1}|T_nf(t)|.$$
A function of the form $Tf$ is usually called a \emph{maximal function}. We shall refer to $T$ as the \emph{maximal operator} of the sequence $(T_n)_{n\geq 1}$. These operators occur naturally in many situations in harmonic analysis and probability theory. One is often interested to show that a maximal operator defines a bounded sublinear operator on an $L^p$-space or, more generally, on a Banach function space. For instance, the celebrated Doob maximal inequality states that for any increasing sequence of conditional expectations and any $1<p\leq\infty$,
\begin{equation}
\label{eqn:DoobClas} \Big\|\sup_{n\geq 1}\cE_n(f)\Big\|_{L^p} \leq c_p \|f\|_{L^p}.
\end{equation}
A fruitful strategy to prove maximal inequalities with sharp constants is to first prove weak type estimates for the maximal operator. Let $(\cA,\Si,\nu)$ be a $\si$-finite measure space. Recall that a sublinear operator $T$ is of Marcinkiewicz weak type (or \emph{M-weak type}) $(p,p)$ if for any $f \in L^p(\cA)$,
\begin{equation}
\label{eqn:M-weakTypeClas}
[\nu(|Tf|>v)]^{\frac{1}{p}} \leq C v^{-1}\|f\|_{L^p(\cA)} \qquad (v>0),
\end{equation}
If $T$ is bounded in $L^{p}$ then $T$ is said to be of strong type $(p,p)$. To prove (\ref{eqn:DoobClas}) one may first show that $\sup_{n\geq 1}\cE_n$ is of M-weak type $(1,1)$ and strong type $(\infty,\infty)$ and then apply Marcinkiewicz' interpolation theorem. This yields a constant $c_p$ of optimal order $\cO((p-1)^{-1})$ as $p\downarrow 1$. More generally, Boyd's interpolation theorem implies that $\sup_{n\geq 1}\cE_n$ is bounded on any symmetric Banach function space with lower Boyd index $p_E>1$. Weak type interpolation has proven effective for a host of maximal inequalities (see e.g.\ \cite{BeS88,Ste93} for examples).\par
In this paper we deal with interpolation questions for maximal operators of weak type which correspond to a sequence $(T_n)_{n\geq 1}$ of positive, sublinear operators on noncommutative symmetric spaces associated with a semi-finite von Neumann algebra $\cM$. In this setting, the maximal operator $\sup_{n\geq 1}T_n$ is a fictitious object: for example, if $x \in L^p(\cM)_+$ then $\sup_{n\geq 1}T_n(x)$ may not exist in terms of the standard ordering of $L^p(\cM)_+$. Indeed, even for two positive semi-definite $n\ti n$ matrices $x,y$ there may not be a matrix $z$ satisfying
$$\langle z\xi,\xi\rangle = \max\{\langle x\xi,\xi\rangle, \langle y\xi,\xi\rangle\} \qquad (\xi \in \C^n).$$
Even though $\sup_{n\geq 1}T_n(x)$ is not well-defined as an operator, one can still make sense of the quantity `$\|\sup_{n\geq 1}T_n(x)\|_p$' by viewing the sequence $(T_n(x))_{n\geq 1}$ as an element of the noncommutative vector-valued space $L^p(\cM;l^{\infty})$, as
introduced by Pisier \cite{Pis98}. With this point of view, Junge \cite{Jun02} showed that one can obtain a noncommutative extension of (\ref{eqn:DoobClas}). He proved that for any increasing sequence of conditional expectations in $\cM$ and $1<p\leq \infty$,
\begin{equation}
\label{eqn:DoobNC} \|(\cE_n(x))_{n\geq 1}\|_{L^p(\cM;l^{\infty})} \leq C_p \|x\|_{L^p(\cM)}.
\end{equation}
More recently, Junge and Xu \cite{JuX07} proved a Marcinkiewicz interpolation theorem for noncommutative maximal operators (see Theorem~\ref{thm:JuXmain} below). This result allows one to prove (\ref{eqn:DoobNC}) elegantly by interpolating between the M-weak type $(1,1)$-maximal inequality for conditional expectations obtained by Cuculescu \cite{Cuc71} and the trivial case $p=\infty$. This approach yields a constant $C_p$ of order $\mathcal{O}((p-1)^{-2})$ when $p\downarrow 1$, which is known to be optimal \cite{JuX05}. The difference between the optimal order of the constants in the classical and noncommutative Doob maximal inequalities underlines the fact that, in general, the extension of maximal inequalities to the noncommutative setting requires nontrivial new ideas.\par
The purpose of the present paper is to prove interpolation results for noncommutative maximal operators of \emph{restricted} weak type. In the classical case, this means that (\ref{eqn:M-weakTypeClas}) is only required to hold for indicator functions $f=\chi_A$, where $A$ is any set of finite measure. Stein and Weiss \cite{StW59} showed that Marcinkiewicz' interpolation theorem remains valid under this relaxed notion of weak type. This extension has proven especially useful for interpolation problems in harmonic analysis, where the weak-type condition (\ref{eqn:M-weakTypeClas}) is typically hard to verify (see e.g.\ \cite{BeS88,StW59}). We expect that the results proved in this paper will be similarly useful in a noncommutative context.\par
The main result of this paper is the following Boyd-type interpolation theorem, see also Theorem~\ref{thm:maxInt} for a slightly more precise statement. Any unexplained terminology can be found in Sections~\ref{sec:prelim} and \ref{sec:weakType}.
\begin{theorem}
\label{thm:maxIntIntro} Let $1\leq p<q\leq\infty$, let $\cM,\cN$ be semi-finite von Neumann algebras and let $E$ be a fully symmetric Banach function space on $\R_+$ with Boyd indices $p_E\leq q_E$. Suppose that $(T_{\al})_{\al\in A}$ is a net of order preserving, sublinear maps which is of restricted weak types $(p,p)$ and $(q,q)$. If $p<p'<p_E$ and either $q_E<q'<q<\infty$ or $q=\infty$, then for any $x\in E(\cM)_+$,
\begin{equation}
\label{eqn:maxIntGenIntro} \|(T_{\al}(x))_{\al \in A}\|_{E(\cN;l^{\infty})} \leq C_{p,p',q,q'}\|S_{p',q'}\|_{E\rightarrow E} \ \|x\|_{E(\cM)},
\end{equation}
where $S_{p',q'}$ is Calder\'{o}n's operator and $C_{p,p',q,q'}$ is of order $\mathcal{O}((p'-p)^{-1})$ as $p'\downarrow p$ and of order $\mathcal{O}((q-q')^{-1})$ as $q'\uparrow q$.
\end{theorem}
Theorem~\ref{thm:maxIntIntro} unifies and extends two main interpolation results for noncommutative maximal operators in the literature. Firstly, Bekjan, Chen and Os\c{e}kowski \cite{BCO12} proved a special case of our result for maximal operators of M-weak type $(p,p)$ and strong type $(\infty,\infty)$, with a larger, suboptimal interpolation constant. Secondly, specialized to $E=L^r$ for $p<r<q$, our result extends the earlier mentioned interpolation theorem in \cite{JuX07} to noncommutative maximal operators of restricted weak type, see the discussion following Corollary~\ref{cor:Marcinkiewicz} for a detailed comparison. In this case, the constant in (\ref{eqn:maxIntGenIntro}) is of order $\mathcal{O}((r-p)^{-2})$ as $r\downarrow p$ and $\mathcal{O}((q-r)^{-2})$ as $r\uparrow q$. In particular, if $p=1$ and $q=\infty$ then Theorem~\ref{thm:maxIntIntro} implies (\ref{eqn:DoobNC}) with constant of the best possible order. In this sense, the constant in (\ref{eqn:maxIntGenIntro}) is optimal.\par
In Section~\ref{sec:mainResults} we illustrate the use of Theorem~\ref{thm:maxIntIntro} by proving Doob's maximal inequality in noncommutative symmetric Banach function spaces under minimal conditions on the underlying function space, see Corollary~\ref{cor:Doob}. In addition, we provide an interpolation result for the generalized moments of noncommutative maximal operators (Theorem~\ref{thm:intMaxOrl}), which resolves an open question in \cite{BCO12}.

\section{Preliminaries}
\label{sec:prelim}

We start by briefly recalling some relevant terminology. Let $S(\R_+)$ be the linear space of all measurable, a.e.\ finite functions $g$ on $\R_+$ which satisfy $\lambda(|g|>v)<\infty$ for some $v>0$, where $\lambda$ is Lebesgue measure on $\R_+$. For any $g \in S(\R_+)$ let $\mu(g)$ denote its decreasing rearrangement
$$\mu_t(g) = \inf\{v>0 \ : \ \lambda(|g|>v) \leq t\} \qquad (t\geq 0).$$
A normed linear subspace $E$ of $S(\R_+)$ is called a \emph{symmetric Banach function space} if it is complete and if for any $g \in S(\R_+)$ and $h\in E$ satisfying $\mu(g)\leq \mu(h)$, we have $g \in E$ and $\|g\|_E\leq \|h\|_E$. Let $H$ be the Hardy-Littlewood operator
$$Hg(t) = \frac{1}{t} \int_0^t \mu_t(g) \ dt \qquad (g \in S(\R_+)).$$
For $g,h \in S(\R_+)$ we write $g\prec\prec h$ if $Hg\leq Hh$. We say that a symmetric Banach function space $E$ is \emph{fully symmetric} if for any $g \in S(\R_+)$ and $h\in E$ satisfying $g\prec\prec h$, we have $g \in E$ and $\|g\|_E\leq \|h\|_E$.\par
Fix a von Neumann algebra $\cM$ acting on a Hilbert space $(\cH,\langle\cdot,\cdot\rangle)$, which is equipped with a normal, semi-finite, faithful trace $\tr$. Let $S(\tr)$ denote the linear space of all $\tr$-measurable operators and let $S(\tr)_+$ be its positive cone. For any $x \in S(\tr)$ its decreasing rearrangement is defined by
\begin{equation*}
\label{eqn:DRoperator} \mu_t(x) = \inf\{v>0 \ : \ \tr(\lambda_{(v,\infty)}(x))\leq t\} \qquad (t\geq 0),
\end{equation*}
where $\lambda(x)$ denotes the spectral measure of $|x|$. Suppose that $a=\sum_{i=1}^n \al_i e_i$, with $\al_1>\al_2>\ldots>\al_n>0$ and $e_1,\ldots,e_n$ projections in $\cM$ satisfying $e_ie_j=0$ for $i\neq j$ and $\tr(e_i)<\infty$ for $1\leq i\leq n$. As in the commutative case (see \cite{BeS88}), it is elementary to show that
\begin{equation}
\label{eqn:DRprojSum}
\mu(a) = \sum_{j=1}^n \al_j \chi_{[\rho_{j-1},\rho_j)},
\end{equation}
where $\rho_0=0$, $\rho_j = \sum_{i=1}^j \tr(p_i)$, $j=1,\ldots,n$. We will also use the submajorization inequality (see \cite{FaK86}, Theorem 4.4)
\begin{equation}
\label{eqn:subMajAdd}
\mu(x+y)\prec\prec \mu(x) + \mu(y) \qquad (x,y \in S(\tr)).
\end{equation}
If $E$ is a symmetric Banach function space on $\R_+$, then we define the associated noncommutative symmetric space
$$E(\cM) := \{x \in S(\tr): \ \|\mu(x)\|_{E}<\infty\}.$$
The space $E(\cM)$ is a Banach space under the norm $\|x\|_{E(\cM)}:=\|\mu(x)\|_{E}$. For $E=L^p$ this construction yields the usual noncommutative $L^p$-spaces.\par
For any directed set $A$ we let $E(\cM;l^{\infty}(A))_+$ denote the set of all nets $x=(x_{\alpha})_{\alpha \in A}$ in $E(\cM)_+$ for which there exists an $a \in E(\cM)_+$ such that $x_{\alpha}\leq a$ for all $\alpha \in A$. For these elements we set
$$\|x\|_{E(\cM;l^{\infty})} := \inf\{\|a\|_{E(\cM)} \ : \ x_{\alpha}\leq a \ \mathrm{for \ all \ \alpha\in A}\}.$$
One may show that, up to a constant depending only on $E$, this expression coincides with the norm of the noncommutative symmetric $l^{\infty}$-valued space $E(\cM;l^{\infty})$ introduced in \cite{Dir12}.\par
Let us finally fix the following notation. For any set $A$ we let $\chi_A$ be its indicator. Also, we write $u\lesssim_{\alpha}v$ if $u\leq c_{\alpha} v$ for some constant $c_{\alpha}$ depending \emph{only} on $\alpha$.

\section{Three flavors of weak type}
\label{sec:weakType}

Before defining the different types of maximal operators, let us recall the classical notions of weak type and restricted weak type for a sublinear operator $T$ on $S(\R_+)$. For $0<p,q\leq\infty$ let $L^{p,q}$ denote the Lorentz spaces on $\R_+$, i.e., the subspace of all $g$ in $S(\R_+)$ such
that
\begin{align*}
\|g\|_{L^{p,q}} = \left\{\begin{array}{rl}(\int_0^{\infty}t^{\frac{q}{p}-1}\mu_t(g)^q \ dt)^{\frac{1}{q}} \qquad & (0<q<\infty),\\
\sup_{0<t<\infty} t^{\frac{1}{p}}\mu_t(g) \qquad & (q=\infty),
\end{array}
\right.
\end{align*}
is finite. Given $0<p<\infty$, we say that $T$ is of weak type $(p,p)$ if there is a constant $C_p>0$ such that for any $g \in L^{p,1}(\R_+)$,
\begin{equation}
\label{eqn:defWTClas}
\lambda(|Tg|>v)^{\frac{1}{p}}\leq C_p v^{-1}\|g\|_{L^{p,1}(\R_+)}.
\end{equation}
An operator $T$ is of restricted weak type $(p,p)$ (as introduced by Stein and Weiss in \cite{StW59}) if (\ref{eqn:defWTClas}) holds only for indicators $g=\chi_A$, where $A$ is any measurable set of finite measure. As is well known, for a given $0<p<\infty$,
$$\mathrm{strong \ type} \ \Rightarrow \ \mathrm{M\small{-}weak \ type} \ \Rightarrow \ \mathrm{weak \ type} \ \Rightarrow \ \mathrm{restricted \ weak \ type}$$
and the reverse implications do not hold in general.\par
For our discussion below we recall the following characterization of weak type operators due to Calder\'{o}n. For $0<p<q<\infty$ we define \emph{Calder\'{o}n's operator} $S_{p,q}$ as the linear operator
$$S_{p,q}g(t) = \tfrac{1}{p}t^{-\frac{1}{p}}\int_0^t s^{\frac{1}{p}}g(s) \frac{ds}{s} + \tfrac{1}{q}t^{-\frac{1}{q}} \int_t^{\infty}s^{\frac{1}{q}}g(s) \frac{ds}{s} \ \qquad (t>0, \ g\in S(\R_+))$$
and for $0<p<\infty$ we set
$$S_{p,\infty}g(t) = \tfrac{1}{p}t^{-\frac{1}{p}}\int_0^t s^{\frac{1}{p}}g(s) \frac{ds}{s} \qquad (t>0, \ g\in S(\R_+)).$$
In \cite{Cal66} Calder\'{o}n proved that a sublinear operator $T$ on $S(\R_+)$ is simultaneously of weak types $(p,p)$ and $(q,q)$ if and only if it satisfies
\begin{equation}
\label{eqn:caldDomClas}
\mu_t(Tg) \lesssim_{p,q} \big(S_{p,q}\mu(g)\big)(t) \qquad (\mathrm{for \ all } \ g \in S(\R_+)).
\end{equation}
In fact, Calder\'{o}n's proof shows that $T$ is of restricted weak types $(p,p)$ and $(q,q)$ if and only if it satisfies
\begin{equation}
\label{eqn:caldDomClasRes}
\mu_t(T\chi_A) \lesssim_{p,q} \big(S_{p,q}\mu(\chi_A)\big)(t),
\end{equation}
for any measurable set $A$ of finite measure.\par
We now extend these definitions to noncommutative maximal operators. Throughout, let $\cM$ and $\cN$ be von Neumann algebras equipped with normal, semi-finite, faithful traces $\tr$ and $\si$, respectively. For any projection $e$ we set $e^{\perp}:=1-e$. Also, if $f$ is another projection, then we let $e\vee f$ and $e\wedge f$ denote the supremum and infimum, respectively, of $e$ and $f$.
\begin{definition}
\label{def:weakType} \emph{For any $0<r<\infty$ we say that a net $(T_{\al})_{\al\in A}$ of maps $T_{\al}:L^{r}(\cM)_+\rightarrow S(\si)_+$ is of} M-weak type $(r,r)$ \emph{if there is a constant $C_r>0$ such that for any $x \in L^r(\cM)_+$ and any $\theta>0$, there exists a projection
$e^{(\theta)}=e_{x}^{(\theta)}$ satisfying}
\begin{equation}
\label{eqn:M-weakType}
\si((e^{(\theta)})^{\perp}) \leq (C_r\theta^{-1})^r\|x\|_{L^r(\cM)}^r \ \ \mathrm{and} \ \ e^{(\theta)}T_{\al}(x)e^{(\theta)}\leq \theta, \ \mathrm{for \ all} \ \al \in A.
\end{equation}
\emph{A net $(T_{\al})_{\al\in A}$ of maps $T_{\al}:L^{r,1}(\cM)_+\rightarrow S(\si)_+$ is of} restricted weak type $(r,r)$ \emph{if there is a constant $C_r>0$ such that for any projection $f$ in $L^{r,1}(\cM)_+$ and any $\theta>0$, there is a projection $e^{(\theta)}=e_{f}^{(\theta)}$ such that}
\begin{equation}
\label{eqn:weakType} \si((e^{(\theta)})^{\perp}) \leq (C_r\theta^{-1})^r\tr(f) \ \ \mathrm{and} \ \ e^{(\theta)}T_{\al}(f)e^{(\theta)}\leq \theta, \ \mathrm{for \ all} \
\al \in A.
\end{equation}
\emph{A net $(T_{\al})_{\al \in A}$ of maps $T_{\al}:\cM_+\rightarrow\cN_+$ is of} restricted weak type $(\infty,\infty)$ \emph{if there is a constant
$C_{\infty}>0$ such that for any projection $f$ in $\cM$,}
$$\sup_{\al \in A}\|T_{\al}(f)\|_{\infty} \leq C_{\infty}.$$
\emph{A net $(T_{\al})_{\al\in A}$ of maps $T_{\al}:L^{r}(\cM)_+\rightarrow S(\si)_+$ is of} strong
type $(r,r)$ \emph{if}
$$\|(T_{\al}(x))_{\al \in A}\|_{L^r(\cN;l^{\infty})} \leq C_r \|x\|_{L^r(\cM)}.$$
\end{definition}
In the commutative case, a sequence $(T_{n})_{n\geq 1}$ is of restricted weak type $(r,r)$ in the sense of Definition~\ref{def:weakType} if $\sup_{n\geq 1}T_{n}$ is of restricted weak type $(r,r)$ in the classical sense. Indeed, in this case one may take $e^{(\theta)}=\chi_{[0,\theta]}(\sup_{n\geq 1}T_{n}(f))$. Thus, loosely speaking, (\ref{eqn:weakType}) states that the fictitious noncommutative maximal operator `$\sup_{\al}T_{\al}$' is of restricted weak type $(r,r)$.
\begin{remark}
\label{rem:M-weaktype}
\emph{In the noncommutative literature (e.g.\ in \cite{BCO12,JuX07}), it has become customary to refer to property (\ref{eqn:M-weakType}) as} weak type\emph{, instead of M-weak type. The terminology used here is in accordance with the classical literature on interpolation theory.}
\end{remark}
We will often assume that the maps $T_{\al}$ satisfy additional properties. We call a map $T:S(\tr)_+\rightarrow S(\si)_+$ \emph{sublinear} if
$$T(cx+dy) \leq cT(x) + dT(y) \qquad (c,d \in \R_+, x,y\in S(\tr)_+)$$
and \emph{order preserving} if $T(x)\leq T(y)$ whenever $x\leq y$ in $S(\tr)_+$.\par
Using real interpolation and duality techniques Junge and Xu proved the following version of Marcinkiewicz' interpolation theorem for noncommutative maximal operators.
\begin{theorem}
\label{thm:JuXmain}
(\cite{JuX07}, Theorem 3.1) Let $1\leq p<q\leq\infty$. If a net $(T_{\al})_{\al\in A}$ of positive, subadditive maps $(T_{\al})$ is of M-weak type $(p,p)$ and strong type $(q,q)$, then for any $p<r<q$,
\begin{equation}
\label{eqn:JuXmain}
\|(T_{\al}(x))_{\al \in A}\|_{L^r(\cN;l^{\infty})} \lesssim C_p^{1-\theta}C_q^{\theta}\Big(\frac{rp}{r-p}\Big)^2 \ \|x\|_{L^r(\cM)},
\end{equation}
where $\theta$ is chosen such that $1/r = (1-\theta)/p + \theta/q$.
\end{theorem}
In Corollary~\ref{cor:Marcinkiewicz} below we will obtain an extension of this result to maximal operators of restricted weak types $(p,p)$ and $(q,q)$.

\section{A Calder\'{o}n-type bound for maximal operators of weak type}

Our starting point for the proof of Theorem~\ref{thm:maxIntIntro} is the characterization (\ref{eqn:caldDomClasRes}). A first thought is to attempt to establish the following direct generalization: if $(T_{\al})$ is of restricted weak types $(p,p)$ and $(q,q)$ then one may try to find, for every projection $f$, a positive measurable operator $a$ satisfying $T_{\al}(f)\leq a$ for all $\al$ and
$$\mu_t(a) \lesssim_{p,q} \big(S_{p,q}\mu(f)\big)(t) \qquad (t>0).$$
Unfortunately, if $\cN$ is noncommutative then this assertion is in general false. Indeed, it would imply that the noncommutative Doob maximal inequality (\ref{eqn:DoobNC}) holds with constant of order $\mathcal{O}((p-1)^{-1})$ for $p\downarrow 1$, whereas the order $\mathcal{O}((p-1)^{-2})$ is known to be optimal \cite{JuX05}. However, we can find a bound of the form
$$\mu_t(a) \leq C_{p,p',q,q'} \big(S_{p',q'}\mu(f)\big)(t),$$
where $p<p'<q'<q$ and $C_{p,p',q,q'}$ is singular as $p'\downarrow p$ and $q'\uparrow q$.\par
For $0<p,q<\infty$ we introduce the constants
\begin{equation}
\label{eqn:kappapq}
\begin{split}
\kappa_{p,q} & = 2^{\frac{1}{p}}\max\{C_p(1-2^{-p})^{-\frac{1}{p}},C_q(1-2^{-q})^{-\frac{1}{q}}\}, \\
\kappa_{p,\infty} & = \max\{C_p(1-2^{-p})^{-\frac{1}{p}}, C_{\infty}\},
\end{split}
\end{equation}
where $C_p$ and $C_q$ are the restricted weak type $(p,p)$ and $(q,q)$ constants in (\ref{eqn:weakType}). If $p,q\geq 1$ then $\kappa_{p,q}\leq 4\max\{C_p,C_q\}$ and $\kappa_{p,\infty}\leq 2\max\{C_p,C_\infty\}$. Also, for $p<p'<q'<q<\infty$ we set
$$\gamma_{p,p'} = \sum_{k\leq 0}2^{(k-1)(1-\frac{p}{p'})}, \qquad \delta_{q,q'}=\sum_{k>0}2^{(k-1)(1-\frac{q}{q'})}.$$
Note that
$$\gamma_{p,p'} = \cO((p'-p)^{-1}) \ \mathrm{as \ } p'\downarrow p, \qquad \delta_{q,q'} = \cO((q-q')^{-1}) \ \mathrm{as \ } q'\uparrow q.$$
\begin{lemma}
\label{lem:majorant} Fix $0<p<q\leq\infty$. Let $(T_{\al})_{\al \in A}$ be a net of positive maps which is of restricted weak types $(p,p)$ and $(q,q)$. Let $p<p'<q'<q$. If $q<\infty$
and $f$ is a projection in $L^{p,1}(\cM)_+ + L^{q,1}(\cM)_+$, then there exists a constant $K_{p,p',q,q'}$ and an $a \in S(\si)_+$ such that
\begin{equation}
\label{eqn:majProp1} T_{\al}(f)\leq a \qquad (\al \in A)
\end{equation}
and
\begin{equation}
\label{eqn:majProp2} \mu_t(a) \leq \kappa_{p,q}K_{p,p',q,q'} \ S_{p',q'}\mu(f)(t)\qquad (t>0).
\end{equation}
Moreover, the constant $K_{p,p',q,q'}$ satisfies
\begin{align}
\label{eqn:majPropConst}
K_{p,p',q,q'} & \leq 4\max\{\gamma_{p,p'},\delta_{q,q'}\}.
\end{align}
If $q=\infty$, then for every projection $f \in \cM_+$ there exists a constant $K_{p,p',\infty,\infty}\leq 4\gamma_{p,p'}$ and an $a \in \cN_+$ satisfying (\ref{eqn:majProp1}) and
\begin{equation}
\label{eqn:majProp2Infty} \mu_t(a) \leq \kappa_{p,\infty}K_{p,p',\infty,\infty} \ S_{p',\infty}\mu(f)(t)\qquad (t>0).
\end{equation}
\end{lemma}
\begin{proof}
We use the functions $\theta_{p,q}:\R_+\rightarrow\R_+$, which for $0<p,q\leq\infty$ are defined by
\begin{align*}
\theta_{p,q}(t) & = t^{-\frac{1}{p}}\chi_{[1,\infty)}(t) + t^{-\frac{1}{q}}\chi_{(0,1)}(t) \qquad (t>0)
\end{align*}
and for $0<p<\infty$ given by
$$\theta_{p,\infty}(t) = t^{-\frac{1}{p}}\chi_{[1,\infty)}(t) + \chi_{(0,1)}(t) \qquad (t>0).$$
We first prove the result for $q<\infty$. In this case $\tr(f)<\infty$. Using the change of variable $s=\tfrac{t}{u}$, we find
\begin{eqnarray}
\label{eqn:thetapqSpq}
S_{p,q}\mu(f)(t) & = & \int_0^1 \mu_{\frac{t}{u}}(f)\tfrac{1}{q}u^{-1-\frac{1}{q}}du + \int_1^{\infty}\mu_{\frac{t}{u}}(f)\tfrac{1}{p}u^{-1-\frac{1}{p}}du \\
& = & \int_0^{\infty}\chi_{(0,\tr(f)]}\Big(\frac{t}{u}\Big)(-\theta_{p,q}'(u))du \nonumber \\
& = & \theta_{p,q}\Big(\frac{t}{\tr(f)}\Big). \nonumber
\end{eqnarray}
Thus, we need to find an $a \in S(\si)_+$ satisfying (\ref{eqn:majProp1}) and
\begin{equation}
\label{eqn:majProp3} \mu_t(a) \leq \kappa_{p,q}K_{p,p',q,q'}\theta_{p',q'}\Big(\frac{t}{\tr(f)}\Big) \qquad (t>0).
\end{equation}
Let us first assume that $\kappa_{p,q}\leq 1$. For any $\theta>1$ fix a projection $e_q^{(\theta)}$ satisfying (\ref{eqn:weakType}) for $r=q$ and for $0<\theta\leq 1$ we
pick $e_p^{(\theta)}$ such that (\ref{eqn:weakType}) holds for $r=p$. For every $k\in\Z$ we define
$$e_k = \Big(\bigwedge_{l\geq k} e_q^{(2^l)}\Big) \ \ \ (k>0), \qquad
e_k = \Big(\bigwedge_{0\geq l\geq k} e_p^{(2^l)}\Big) \wedge \Big(\bigwedge_{l\geq 0} e_q^{(2^l)}\Big) \ \ \ (k\leq 0)$$ and we set
$$d_k = e_k - e_{k-1}.$$
Observe that $(e_k)_{k\in \Z}$ is increasing, and therefore $d_kd_l=0$ for $k\neq l$ and $d_k^2=d_k$. Note that $e_k \leq e_q^{(2^k)}$ for $k>0$ and $e_k \leq e_p^{(2^k)}$ for $k\leq 0$, hence $e_kT_{\al}(x)e_k\leq 2^k$ for all $k$. If $k>0$ then, using $\kappa_{p,q}\leq 1$,
\begin{eqnarray}
\label{eqn:perpEst1} \si(e_k^{\perp}) & \leq & \sum_{l\geq k} \si((e_q^{(2^l)})^{\perp}) \\
& \leq & \sum_{l\geq k} C_q^q 2^{-lq} \tr(f) \leq 2^{-kq}C_q^q\frac{1}{1-2^{-q}}\tr(f)\leq 2^{-kq}\tr(f),\nonumber
\end{eqnarray}
and in particular it follows that $e_k\uparrow 1$ for $k\rightarrow\infty$. Moreover, if $k\leq 0$, then again using $\kappa_{p,q}\leq 1$,
\begin{eqnarray}
\label{eqn:perpEst2} \si(e_k^{\perp}) & \leq & \sum_{0\geq l\geq k} \si((e_p^{(2^l)})^{\perp}) + \sum_{l\geq 0} \si((e_q^{(2^l)})^{\perp}) \\
& \leq & 2^{-kp}(C_p^p(1-2^{-p})^{-1}+ C_q^q(1-2^{-q})^{-1})\tr(f) \leq 2^{-kp}\tr(f). \nonumber
\end{eqnarray}
Finally, we set $e_{-\infty}:=\wedge_{k\leq 0} \ e_k$. Since $e_k\uparrow 1$,
\begin{equation}
\label{eqn:e0e0perp}
e_0 = \sum_{k\leq 0} d_k + e_{-\infty}, \qquad e_0^{\perp} = \sum_{k>0}d_k.
\end{equation}
Set $K_{p,p',q,q'} = 4\max\{\gamma_{p,p'},\delta_{q,q'}\}$ and let $(a_N)_{N\geq 1}$ be the sequence in $\cM_+$ given by
$$a_N = K_{p,p',q,q'}\Big(\sum_{-\infty<k\leq 0}2^{(k-1)p/p'} d_k + \sum_{0<k\leq N} 2^{(k-1)q/q'}d_k\Big).$$
As our candidate for the sought operator $a \in S(\si)_+$ we would like to define $a:=\lim_{N\rightarrow\infty} a_N$. To show that this limit exists in $S(\si)$, we will first show that the estimate (\ref{eqn:majProp3}) is satisfied for $a=a_N$, uniformly in $N$. Since the coefficients of the $d_k$ are increasing, we find using (\ref{eqn:DRprojSum}) that
\begin{align*}
\mu(a_N) & = K_{p,p',q,q'}\Big(\sum_{k\leq 0} 2^{(k-1)p/p'} \chi_{[\si(e_N-e_k),\si(e_N-e_{k-1}))} \\
& \ \ \ \ + \sum_{0<k\leq N-1} 2^{(k-1)q/q'} \chi_{[\si(e_N-e_k),\si(e_N-e_{k-1}))} + 2^{(N-1)q/q'}\chi_{[0,\si(e_N-e_{N-1}))}\Big) \\
& \leq K_{p,p',q,q'}\Big(\sum_{k\leq 0} 2^{(k-1)p/p'} \chi_{[\si(1-e_k),\si(1-e_{k-1}))} \\
& \ \ \ \ + \sum_{0<k\leq N-1} 2^{(k-1)q/q'} \chi_{[\si(1-e_k),\si(1-e_{k-1}))} + 2^{(N-1)q/q'}\chi_{[0,\si(1-e_{N-1}))}\Big)
\end{align*}
and by applying (\ref{eqn:perpEst1}) and (\ref{eqn:perpEst2}) it follows that
\begin{align*}
& \mu(a_N)\leq K_{p,p',q,q'}\Big(\sum_{k\leq 0}2^{(k-1)p/p'} \chi_{[2^{-kp}\tr(f),2^{-(k-1)p}\tr(f))} \\
& \qquad \qquad + \sum_{0<k<N}2^{(k-1)q/q'} \chi_{[2^{-kq}\tr(f),2^{-(k-1)q}\tr(f))} + 2^{(N-1)q/q'}\chi_{[0,2^{-(N-1)q}\tr(f))}\Big).
\end{align*}
If $k\leq 0$ and $2^{-kp}\tr(f)\leq t< 2^{-(k-1)p}\tr(f)$, then
$$2^{k-1} < \Big(\frac{t}{\tr(f)}\Big)^{-\frac{1}{p}}$$
and therefore
$$2^{(k-1)p/p'} < \Big(\frac{t}{\tr(f)}\Big)^{-\frac{1}{p'}}$$
For all $k>0$ and $2^{-kq}\tr(f)\leq t< 2^{-(k-1)q}\tr(f)$ we find
$$2^{k-1} < \Big(\frac{t}{\tr(f)}\Big)^{-\frac{1}{q}}$$
and so
$$2^{(k-1)q/q'} < \Big(\frac{t}{\tr(f)}\Big)^{-\frac{1}{q'}}.$$
We conclude that for any $t>0$,
\begin{align*}
\mu_t(a_N) & \leq K_{p,p',q,q'}\Big(\Big(\frac{t}{\tr(f)}\Big)^{-\frac{1}{p'}}\chi_{[\tr(f),\infty)}(t) + \Big(\frac{t}{\tr(f)}\Big)^{-\frac{1}{q'}}\chi_{(0,\tr(f))}(t)\Big) \\
& \leq K_{p,p',q,q'}
\theta_{p',q'}\Big(\frac{t}{\tr(f)}\Big).
\end{align*}
Since this inequality holds uniformly in $N$, we conclude that $(a_N)_{N\geq 1}$ is an increasing sequence which is bounded in measure. Hence, by \cite{Pag07}, Theorem
5.10, there exists an $a \in S(\si)_+$ such that $a_N\uparrow a$ in $S(\si)_+$. We claim that $a$ has the asserted properties. Since $\mu(a_N)\uparrow \mu(a)$
(\cite{Pag07}, Proposition 6.5) it is clear that (\ref{eqn:majProp3}) holds. Since $T_{\al}(f)\geq 0$, we know that (see e.g.\ \cite{Dir12}, Lemma 5.9)
$$T_{\al}(f) \leq 2e_0T_{\al}(f)e_0 + 2e_0^{\perp}T_{\al}(f)e_0^{\perp}.$$
For any $\xi$ in the domain $D(a^{\frac{1}{2}})$ of $a^{\frac{1}{2}}$ we have
$$\|a^{\frac{1}{2}}\xi\|^2 = \lim_{N\rightarrow\infty}\langle a_N\xi,\xi\rangle = K_{p,p',q,q'}\Big(\sum_{k\leq 0}2^{(k-1)p/p'} \langle d_k\xi,\xi\rangle + \sum_{k>0}2^{(k-1)q/q'} \langle d_k\xi,\xi\rangle\Big).$$
Notice that $e_{-\infty}T_{\al}(f)e_{-\infty} = e_{-\infty}e_kT_{\al}(f)e_ke_{-\infty}\leq 2^{k}e_{-\infty}$ for all $k\leq 0$ and therefore
$e_{-\infty}T_{\al}(f)e_{-\infty}=0$. Moreover,
$$\|e_0T_{\al}(f)e_{-\infty}\|_{\infty} \leq \|e_0T_{\al}(f)e_0\|_{\infty}^{\frac{1}{2}} \|e_{-\infty}T_{\al}(f)e_{-\infty}\|_{\infty}^{\frac{1}{2}} = 0.$$
Together with (\ref{eqn:e0e0perp}) this implies that any $\xi \in D(a^{\frac{1}{2}})$ satisfies
\begin{eqnarray}
\label{eqn:inprodEst}
\langle e_0T_{\al}(f)e_0\xi,\xi\rangle
& = & \Big\langle \Big(\sum_{k\leq 0}d_k\Big) T_{\al}(f)\Big(\sum_{l\leq 0}d_k\Big)\xi,\xi\Big\rangle \nonumber \\
& \leq & \sum_{k,l\leq 0}\|d_kT_{\al}(f)d_l\|_{\infty} \ \|d_k\xi\| \ \|d_l\xi\| \nonumber \\
& \leq & \sum_{k,l\leq 0} \|d_kT_{\al}(f)d_k\|_{\infty}^{\frac{1}{2}} \ \|d_lT_{\al}(f)d_l\|_{\infty}^{\frac{1}{2}} \ \|d_k\xi\| \ \|d_l\xi\| \nonumber\\
& = & \Big(\sum_{k\leq 0} \|d_kT_{\al}(f)d_k\|_{\infty}^{\frac{1}{2}} \ \|d_k\xi\|\Big)^2.
\end{eqnarray}
Since
$$\|d_kT_{\al}(f)d_k\|_{\infty}^{\frac{1}{2}} \leq 2^{\frac{k}{2}} = 2^{1/2}2^{(k-1)p/2p'}2^{(k-1)(1-\frac{p}{p'})/2}$$
we find by applying the Cauchy-Schwarz inequality in (\ref{eqn:inprodEst}),
$$\langle e_0T_{\al}(f)e_0\xi,\xi\rangle \leq 2\gamma_{p,p'} \sum_{k\leq 0} 2^{(k-1)p/p'}\|d_k\xi\|^2 = \Big\langle \sum_{k\leq 0} 2\gamma_{p,p'}2^{(k-1)p/p'}d_k\xi,\xi\Big\rangle.$$
By similar reasoning,
$$\langle e_0^{\perp}T_{\al}(f)e_0^{\perp}\xi,\xi\rangle \leq \Big\langle\sum_{k>0}2\delta_{q,q'}2^{(k-1)q/q'} d_k\xi,\xi\Big\rangle.$$
Putting these estimates together we conclude that $\xi \in D(T_{\al}(f)^{\frac{1}{2}})$ and
$$\langle T_{\al}(f)\xi,\xi\rangle \leq \sum_{k\leq 0}4\gamma_{p,p'}2^{(k-1)p/p'} \langle d_k\xi,\xi\rangle +  \sum_{k>0}4\delta_{q,q'}2^{(k-1)q/q'}\langle d_k\xi,\xi\rangle \leq \langle a\xi,\xi\rangle,$$
which establishes (\ref{eqn:majProp1}) (cf.\ \cite{Pag07}, Proposition 4.5). This completes the proof in the case $q<\infty$ under the additional assumption
$\kappa_{p,q}\leq 1$.\par In the general case, define $\tilde{T}_{\al}(f)=\kappa_{p,q}^{-1}T_{\al}(f)$. If $e^{(\theta)}$ satisfies (\ref{eqn:weakType}), then
$\tilde{e}^{(\theta)}:=e^{(\kappa_{p,q}\theta)}$ satisfies, for $r=p,q$,
$$\tr((\tilde{e}^{(\theta)})^{\perp}) \leq (C_r\kappa_{p,q}^{-1}\theta^{-1})^r\tr(f) \ \ \mathrm{and}
 \ \ \tilde{e}^{(\theta)}\tilde{T}_{\al}(f)\tilde{e}^{(\theta)}\leq \theta, \ \mathrm{for \ all} \ \al \in A.$$
Therefore, $\tilde{\kappa}_{p,q}\leq 1$ and by the above we find an $\tilde{a}\in S(\si)_+$ such that $\tilde{T}_{\al}(f)\leq \tilde{a}$ for all $\al \in A$ and
$$\mu_t(\tilde{a}) \leq K_{p,p',q,q'}S_{p,q}\mu(f)(t) \qquad (t>0).$$
The operator $a:=\kappa_{p,q}\tilde{a}$ has the desired properties.\par
Suppose now that $q=\infty$. Let us first note that if $\tr(f)=\infty$, then $\mu(f) = \chi_{[0,\infty)}$ and we can take $a=C_{\infty}1$. If $\tr(f)<\infty$, then we may assume that $\kappa_{p,\infty}\leq 1$. For $k\leq 0$ we set
$$e_k = \Big(\bigwedge_{0\geq l\geq k} e_p^{(2^l)}\Big)$$
and let $d_k=e_k-e_{k-1}$ as before. We define $a \in \cN_+$ to be the operator
$$a = 2 e_0^{\perp} + \sum_{k\leq 0}4\gamma_{p,p'}2^{(k-1)p/p'} d_k.$$
By following the argument presented above one shows that $a$ satisfies (\ref{eqn:majProp1}) and (\ref{eqn:majProp2}). The details are left to the reader.
\end{proof}
\begin{remark}
\emph{If $\cN$ is commutative, then one can show using essentially the same arguments that $a\in S(\si)_+$ defined by
$$a = \sum_{k \in \Z}2^{k+1} d_k$$
satisfies (\ref{eqn:majProp1}) and
\begin{equation*}
\label{eqn:majProp2com}
\mu_t(a) \leq 4\kappa_{p,q} S_{p,q}\mu(f)(t)\qquad (t>0).
\end{equation*}
In this case one uses that $d_kT_{\al}(f)d_l=0$ whenever $k\neq l$.}
\end{remark}
In order to obtain interpolation results for noncommutative maximal operators, we need to extend Lemma~\ref{lem:majorant} from projections to arbitrary measurable operators. We achieve this by representing a given measurable operator $x$ as a series of weighted projections and applying
Lemma~\ref{lem:majorant} term-wise.
\begin{theorem}
\label{thm:CaldDom} Fix $1\leq p<q\leq\infty$. Let $(T_{\al})_{\al \in A}$ be a net of order preserving, sublinear operators which is of restricted weak types $(p,p)$ and $(q,q)$. Let $p<p'$ and, if $q<\infty$, we fix $q'<q$. Then for any $x\in L^{p',1}(\cM)_+ + L^{q',1}(\cM)_+$ there exists an $a \in S(\si)_+$ such that $T_{\al}(x)\leq a$ for all $\al \in A$ and
\begin{equation}
\label{eqn:CaldDom} \mu(a) \prec\prec 4\kappa_{p,q}K_{p,p',q,q'} \ S_{p',q'}\mu(x).
\end{equation}
If $q=\infty$, then for any $x\in L^{p',1}(\cM)_+ + \cM_+$ there exists an $a \in S(\si)_+$ satisfying (\ref{eqn:CaldDom}) with $q'=\infty$ and $T_{\al}(x)\leq a$ for all $\al \in A$.
\end{theorem}
\begin{proof}
Suppose first that $x \in L^{p',1}(\cM)$. Let $\lambda(x)$ be the spectral measure of $x$. For any $k\in\Z$, define $f_k = \lambda_{(2^k,\infty)}(x)$ and consider the dyadic discretization $\hat{x} = \sum_{j\in\Z} 2^{j+1} \lambda_{(2^j,2^{j+1}]}(x)$. Clearly, $x\leq\hat{x}\leq 2x$. By summation by parts,
$$\hat{x} = \sum_{j\in \Z} \sum_{k\leq j}2^k \lambda_{(2^j,2^{j+1}]}(x) =
\sum_{k\in \Z} \sum_{j\geq k}2^k \lambda_{(2^j,2^{j+1}]}(x) = \sum_{k\in \Z} 2^k f_k.$$
Also, by (\ref{eqn:DRprojSum}) and summation by parts
\begin{align}
\label{eqn:muTildex}
\mu(\hat{x}) & = \sum_{j\in\Z} 2^{j+1} \chi_{[\tr(f_{j+1}),\tr(f_j))} \nonumber \\
& = \sum_{j\in\Z}\sum_{k\leq j} (2^{k+1}-2^k) \chi_{[\tr(f_{j+1}),\tr(f_j))} = \sum_{k\in\Z}2^k\chi_{[0,\tr(f_k))} = \sum_{k\in\Z}2^k\mu(f_k).
\end{align}
Let $\hat{a}_k \in S(\si)_+$ be the operator obtained by applying Lemma~\ref{lem:majorant} to $f_k$. For $N\geq 1$ define $\hat{x}_N = \sum_{k=-N}^N 2^{k} f_k$ and set $a_N=\sum_{k=-N}^N 2^k\hat{a}_k$. By sublinearity of $T_{\al}$,
$$T_{\al}(\hat{x}_N)\leq \sum_{k=-N}^N 2^kT_{\al}(f_k)\leq a_N.$$
By (\ref{eqn:subMajAdd}),
\begin{equation*}
\mu(a_N) \prec\prec \sum_{k=-N}^N 2^k\mu(\hat{a}_k).
\end{equation*}
Using lemma~\ref{lem:majorant}, linearity of $S_{p',q'}$ and (\ref{eqn:muTildex}) we find for any $t>0$
\begin{align*}
\sum_{k=-N}^N 2^k\mu_t(\hat{a}_k) & \leq \kappa_{p,q}K_{p,p',q,q'}\sum_{k=-N}^N 2^k S_{p',q'}\mu(f_k)(t) \\
& \leq \kappa_{p,q}K_{p,p',q,q'} \ S_{p',q'}\mu(\hat{x})(t) \\
& \leq 2\kappa_{p,q}K_{p,p',q,q'} \ S_{p',q'}\mu(x)(t).
\end{align*}
This shows in particular that $(a_N)_{N\geq 1}$ is increasing and bounded in measure. Therefore, by \cite{Pag07}, Theorem 5.10 there exists an $a \in S(\si)_+$ such that $a_N\uparrow a$ in $S(\si)_+$. Since $\mu(a_N)\uparrow \mu(a)$, we conclude by monotone convergence that (\ref{eqn:CaldDom}) holds. It is clear that $T_{\al}(\hat{x}_N)\leq a$ for all $N\geq 1$. Note that $T_{\al}$ is of M-weak type $(p',p')$, with M-weak type constant bounded by $C_p(p'-1)^{-1}$, by the same argument as in the commutative case (\cite{BeS88}, Theorem 5.3). Therefore,
$$\|T_{\al}(\hat{x}) - T_{\al}(\hat{x}_N)\|_{p',\infty} \leq \|T_{\al}(\hat{x} - \hat{x}_N)\|_{p',\infty} \lesssim_{p,p'} \|\hat{x} - \hat{x}_N\|_{p',1} \rightarrow 0,$$
as $N\rightarrow \infty$ by dominated convergence. In particular, $T_{\al}(\hat{x}_N)\rightarrow T_{\al}(\hat{x})$ in measure. Since $T_{\al}$ is order preserving, we conclude that
\begin{equation}
\label{eqn:domTalBya}
T_{\al}(x) \leq T_{\al}(\hat{x})\leq a.
\end{equation}
The result follows analogously if $x \in L^{q',1}(\cM)$ if $q'<q<\infty$.\par
Suppose now that $x \in \cM_+$ and $q=\infty$. Let $N^*$ be such that $2^{N^*}\leq \|x\|_{\infty} \leq 2^{N^*+1}$ and define for all $N\geq 1$ the operator $a_N=\sum_{-N\leq k\leq N^*}2^k \hat{a}_k$ in $\cN_+$. By the argument above, the operator $a = \lim_{N\rightarrow\infty} a_N$ is well-defined in $S(\si)_+$ and
$$\mu(a)\prec\prec 2\kappa_{p,\infty}K_{p,p',\infty,\infty} \ S_{p',\infty}\mu(x)$$
Since $T_{\al}$ is sublinear and order preserving, we have for any $N\geq 1$,
\begin{align*}
T_{\al}(\hat{x}) & \leq T_{\al}\Big(\sum_{k\leq -N} 2^{k} f_k\Big) + \sum_{k=-N+1}^{N^*} 2^{k} T_{\al}(f_k) \\
& \leq 2^{-N+1} T_{\al}(\lambda_{(0,\infty)}(x))  + \sum_{k\in \Z} 2^{k} T_{\al}(f_k) \leq 2^{-N+1} C_{\infty} \id + a.
\end{align*}
As this holds for all $N\geq 1$, we conclude that again (\ref{eqn:domTalBya}) holds.\par
Finally, let $x=x_1+x_2$ with $x_1 \in L^{p',1}(\cM)_+$ and $x_2 \in L^{q',1}(\cM)_+$ (or $x_2 \in \cM_+$ if $q=\infty$) and let $a_1,a_2 \in S(\si)_+$ be two operators verifying the asserted properties for $x_1,x_2$. Set $a=a_1+a_2$, then $T_{\al}(x)\leq a$ and moreover,
\begin{align*}
\mu(a) & \prec\prec \mu(a_1) + \mu(a_2) \\
& \prec\prec 2\kappa_{p,q}K_{p,p',q,q'} \ S_{p',q'}(\mu(x_1) + \mu(x_2)) \leq 4\kappa_{p,q}K_{p,p',q,q'} \ S_{p',q'}(\mu(x)).
\end{align*}
This concludes the proof.
\end{proof}
\begin{remark}
\emph{One cannot replace (\ref{eqn:CaldDom}) in Theorem~\ref{thm:CaldDom} by the stronger assertion
\begin{equation*}
\mu(a) \leq 4\kappa_{p,q}K_{p,p',q,q'} \ S_{p',q'}\mu(x).
\end{equation*}
Indeed, if $\cN$ is commutative this would mean that
\begin{equation*}
\mu(a) \lesssim_{p,q} S_{p,q}\mu(x),
\end{equation*}
as $K_{p,p',q,q'}$ is not singular for $p'\downarrow p$ or $q'\uparrow q$ in this case. In particular, by Calder\'{o}n's characterization (\ref{eqn:caldDomClas}) this would imply that every maximal operator of restricted weak types $(1,1)$ and $(\infty,\infty)$ is in fact of weak types $(1,1)$ and $(\infty,\infty)$. However, this is not true (see e.g.\ \cite{HaJ05} for a counterexample).}
\end{remark}

\section{Interpolation of noncommutative maximal inequalities}
\label{sec:mainResults}

To extract interpolation results from Theorem~\ref{thm:CaldDom} we recall the fundamental connection, due to Boyd \cite{Boy69}, between Boyd's indices and Calder\'{o}n's operators. The Boyd indices of a symmetric Banach function space $E$ on $\R_+$ are defined by
$$p_E = \lim_{s\rightarrow\infty}\frac{\log s}{\log \|D_{1/s}\|} \ \ \mathrm{and} \ \ q_E = \lim_{s\downarrow 0}\frac{\log s}{\log \|D_{1/s}\|},$$
where $D_{1/s}$ is the dilation operator $(D_{1/s}g)(t) = g(t/s), \ t \in \R_+$.
\begin{theorem}
\label{thm:BoydObsOrig} \cite{Boy69} If $E$ is a symmetric Banach function space on $\R_+$, then the following hold.
\begin{enumerate}
\item If $1\leq p<q<\infty$, then $S_{p,q}$ is bounded on $E$ if and only if $p<p_E\leq q_E<q$.
\item If $1\leq p<\infty$, then $S_{p,\infty}$ is bounded on $E$ if and only if $p<p_E$.
\end{enumerate}
\end{theorem}
By combining Theorems~\ref{thm:CaldDom} and \ref{thm:BoydObsOrig} we obtain the following interpolation theorem, which extends \cite{BCO12}, Theorem 5.4, as well as \cite{Dir12}, Theorem 6.5, to a larger class of noncommutative maximal operators. More importantly, we find a better interpolation constant.
\begin{theorem}
\label{thm:maxInt} Let $1\leq p<q\leq\infty$ and let $E$ be a fully symmetric Banach function space. Let $(T_{\al})_{\al\in A}$ be a net of order preserving, sublinear maps which is of restricted weak types $(p,p)$ and $(q,q)$. Let $\kappa_{p,q}$ be the constant in (\ref{eqn:kappapq}) and $K_{p,p',q,q'}$ be the constant in (\ref{eqn:majPropConst}). If $p<p'<p_E$ and either $q_E<q'<q<\infty$ or $q=\infty$, then for any $x\in E(\cM)_+$,
\begin{equation}
\label{eqn:maxIntGen} \|(T_{\al}(x))_{\al \in A}\|_{E(\cN;l^{\infty})} \leq 4\kappa_{p,q}K_{p,p',q,q'}\|S_{p',q'}\|_{E\rightarrow E} \ \|x\|_{E(\cM)}.
\end{equation}
\end{theorem}
\begin{proof}
If $x \in E(\cM)_+$ then $x \in L^{p',1}(\cM)_+ + L^{q',1}(\cM)_+$ if $q<\infty$ and $x \in L^{p',1}(\cM)_+ + \cM_+$ if $q=\infty$ by the assumptions on $p_E$ and $q_E$. Let $a\in S(\si)_+$ be the operator given by Theorem~\ref{thm:CaldDom}. Since $E$ is fully symmetric, it follows from (\ref{eqn:CaldDom}) and Theorem~\ref{thm:BoydObsOrig} that $a \in E(\cM)_+$ and
$$\|a\|_{E(\cM)}\leq 4\kappa_{p,q}K_{p,p',q,q'}\|S_{p',q'}\mu(x)\|_E \leq 4\kappa_{p,q}K_{p,p',q,q'}\|S_{p',q'}\|_{E\rightarrow E} \ \|x\|_{E(\cM)}.$$
Thus, $(T_{\al}(x))_{\al \in A}$ is in $E(\cN;l^{\infty})$ and (\ref{eqn:maxIntGen}) holds.
\end{proof}
The following Marcinkiewicz-type interpolation theorem for noncommutative maximal operators is a special case of Theorem~\ref{thm:maxInt}.
\begin{corollary}
\label{cor:Marcinkiewicz} Fix $1\leq p<p'<r<q'<q\leq\infty$. If $(T_{\al})_{\al\in A}$ is a net of order preserving, sublinear maps which is simultaneously of restricted weak types $(p,p)$ and $(q,q)$ with constants $C_p$ and $C_q$, then for any $x \in L^r(\cM)_+$,
$$\|(T_{\al}(x))_{\al \in A}\|_{L^r(\cN;l^{\infty})} \lesssim \max\{C_p,C_q\}\Big(\frac{p'}{p'-p} + \frac{q'}{q-q'}\Big)\Big(\frac{r}{r-p'}+\frac{r}{q'-r}\Big) \ \|x\|_{L^r(\cM)}.$$
In particular,
\begin{equation}
\label{eqn:MarcIntGeneral}
\|(T_{\al}(x))_{\al \in A}\|_{L^r(\cN;l^{\infty})} \lesssim \max\{C_p,C_q\}\Big(\frac{rp}{r-p} + \frac{rq}{q-r}\Big)^2 \ \|x\|_{L^r(\cM)}.
\end{equation}
\end{corollary}
\begin{proof}
Using Hardy's inequalities (see e.g.\ \cite{BeS88}, Lemma III.3.9) one readily shows that
\begin{equation}
\label{eqn:HardyIneqBound}
\|S_{p',q'}\|_{L^r\rightarrow L^r} \leq \Big(\frac{r}{r-p'} + \frac{r}{q'-r}\Big).
\end{equation}
Since $p_{L^r}=q_{L^r}=r$, the first assertion is now immediate from Theorem~\ref{thm:maxInt}. Taking $p'=p/2+r/2$ and $q'=r/2+q/2$ yields (\ref{eqn:MarcIntGeneral}).
\end{proof}
Let us compare this result to Theorem~\ref{thm:JuXmain}. Thanks to the strong type assumption on the right endpoint, the interpolation constant in (\ref{eqn:JuXmain}) does not depend on $q$. Also the dependence on the constants $C_p$ and $C_q$ is better than in (\ref{eqn:MarcIntGeneral}). On the other hand, Corollary~\ref{cor:Marcinkiewicz} requires only a restricted weak type assumption on both endpoints, which is easier to verify in practice, and the interpolation constant is of the right order under these conditions. Also, it is clear that the interpolation result for more general noncommutative symmetric spaces in Theorem~\ref{thm:maxInt} cannot be obtained using the real interpolation techniques used in \cite{JuX07}.\par
As an illustration of our interpolation results, we deduce Doob's maximal inequality. This result was obtained in \cite{Jun02,JuX07} for $E=L^p$ and, using a different argument using duality, for more general symmetric Banach function spaces in \cite{Dir12}. The proof here removes some unnecessary assumptions on $E$ from \cite{Dir12}, Theorem 6.7. Note that the assumption $p_{E}>1$ cannot be removed, as Doob's maximal inequality fails if $E=L^1$.
\begin{corollary}
\label{cor:Doob}
Let $\cM$ be a semi-finite von Neumann algebra and let $(\cE_n)_{n\geq 1}$ be an increasing sequence of conditional expectations in $\cM$. If $E$ is a symmetric Banach function space on $\R_+$ with $p_E>1$, then there is a constant $C_E$ depending only on $E$ such that
$$\|(\cE_n(x))_{n\geq 1}\|_{E(\cM;l^{\infty})} \leq C_E \|x\|_{E(\cM)} \qquad (x \in E(\cM)).$$
If $p>1$ then $C_{L^p}$ is of optimal order $\mathcal{O}((p-1)^{-2})$ as $p\downarrow 1$.
\end{corollary}
\begin{proof}
If $p_E>1$, then $E$ is fully symmetric up to a constant, i.e., if $g \in S(\R_+)$ and $h \in E$ satisfy $g\prec\prec h$, then $g \in E$ and $\|g\|_E\lesssim_E \|h\|_E$ (see the proof of Lemma 3.6 in \cite{DPP11}). Since $(\cE_n)_{n\geq 1}$ is of M-weak type $(1,1)$ (cf.\ \cite{Cuc71}) and strong type $(\infty,\infty)$, the result now follows from Theorem~\ref{thm:maxInt} and Corollary~\ref{cor:Marcinkiewicz}.
\end{proof}
To conclude this paper, we deduce an interpolation theorem for the generalized moments of noncommutative maximal operators from Theorem~\ref{thm:CaldDom}. Recall the following definitions. Let $\Phi:[0,\infty)\rightarrow [0,\infty]$ be a convex Orlicz function, i.e., a continuous, convex and increasing function satisfying $\Phi(0)=0$ and $\lim_{t\rightarrow \infty}\Phi(t) = \infty$. The \emph{Orlicz space} $L_{\Phi}$ is the subspace of all $f$ in $S(\R_+)$ such that for some $k>0$,
$$\int_0^{\infty} \Phi\Big(\frac{|f(t)|}{k}\Big) dt<\infty.$$
We may equip $L_{\Phi}$ with the Luxemburg norm
$$\|f\|_{L_{\Phi}} = \inf\Big\{k>0 \ : \ \int_0^{\infty} \Phi\Big(\frac{|f(t)|}{k}\Big) dt\leq 1\Big\}.$$
Under this norm $L_{\Phi}$ is a symmetric Banach function space \cite{BeS88}. We define the indices of $\Phi$ by
\begin{equation*}
\begin{split}
p_{\Phi} & = \sup\Big\{p>0 \ : \ \int_0^t s^{-p}\Phi(s)\frac{ds}{s} \simeq_p t^{-p}\Phi(t) \ \mathrm{for \ all} \ t>0\Big\}, \\
q_{\Phi} & = \inf\Big\{q>0 \ : \ \int_t^{\infty} s^{-q}\Phi(s)\frac{ds}{s} \simeq_q t^{-q}\Phi(t) \ \mathrm{for \ all} \ t>0\Big\}.
\end{split}
\end{equation*}
We say that $\Phi$ satisfies the global $\Delta_2$-condition if for some constant $C>0$,
\begin{equation}
\label{eqn:D2global}
\Phi(2t) \leq C \Phi(t) \qquad (t\geq 0).
\end{equation}
It is known that $p_{\Phi}$ and $q_{\Phi}$ coincide with the Boyd indices of $L_{\Phi}$. Moreover, one may show that (\ref{eqn:D2global}) holds if and only if $q_{\Phi}<\infty$. For more details we refer to the discussion in \cite{Dir12}.
\begin{theorem}
\label{thm:intMaxOrl} Let $1\leq p<q\leq\infty$ and let $\Phi$ be an Orlicz function satisfying the global $\Delta_2$-condition. Let $(T_{\al})_{\al\in A}$ be a net of order preserving, sublinear maps which is of restricted weak types $(p,p)$ and $(q,q)$. If $p<p_{\Phi}$ and either $q_{\Phi}<q<\infty$ or $q=\infty$, then for any $x\in L_{\Phi}(\cM)_+$ there exists an $a \in L_{\Phi}(\cN)_+$ such that $T_{\al}(x)\leq a$ for all $\al \in A$ and
$$\si(\Phi(a)) \lesssim_{p,q,\Phi} \tr(\Phi(x)).$$
\end{theorem}
\begin{proof}
Suppose that $q<\infty$. Fix $p<p'<\tilde{p}<p_{\Phi}\leq q_{\Phi}<\tilde{q}<q'<q$. Let $x\in L_{\Phi}(\cM)_+$ and let $a \in S(\si)_+$ be the operator provided by Theorem~\ref{thm:CaldDom}. We know that
$$\mu(a) \leq H\mu(a) \lesssim_{p,p',q,q'} HS_{p',q'}\mu(x).$$
By (\ref{eqn:HardyIneqBound}), $S_{p',q'}$ and $H=S_{1,\infty}$ are bounded on $L^{\tilde{p}}(\R_+)$ and $L^{\tilde{q}}(\R_+)$. By \cite{Zyg56}, Theorem 2 (see also \cite{Dir12}, Theorem 4.4), we can now conclude that $a\in L_{\Phi}(\cN)_+$ and
$$\si(\Phi(a)) \lesssim_{p,p',q,q',\Phi} \int_0^{\infty}\Phi(HS_{p',q'}\mu(x)(t))dt \lesssim_{p',\tilde{p},q',\tilde{q},\Phi} \int_0^{\infty}\Phi(\mu_t(x))dt = \tr(\Phi(x)).$$
The proof in the case $q=\infty$ is similar.
\end{proof}
In the special case that $(T_{\al})$ is of M-weak type $(p,p)$ and strong type $(\infty,\infty)$, the above result was obtained in \cite{BCO12}, Theorem 3.2. The general case proved here affirmatively answers an open question in this paper (\cite{BCO12}, Remark 3.3 (2)).

\end{document}